\documentclass[11pt, reqno]{amsart}

\usepackage{amsmath, amsfonts, amssymb, amsbsy, bigstrut, graphicx, enumerate,  upref, longtable, comment}

\usepackage{pstricks}

\usepackage[numbers, sort&compress]{natbib} 

\usepackage[breaklinks]{hyperref}

\newcommand{\vol}{\mathrm{Vol}}

\newcommand{\ep}{\epsilon}

\newtheorem{thm}{Theorem}[section]
\newtheorem{lmm}[thm]{Lemma}
\newtheorem{cor}[thm]{Corollary}

\newtheorem{defn}[thm]{Definition}

\newcommand{\bigavg}[1]{\biggl\langle #1 \biggr\rangle}

\newcommand{\ee}{\mathbb{E}}

\newcommand{\mf}{\mathcal{F}}

\newcommand{\ml}{\mathcal{L}}
\newcommand{\cp}{\mathcal{P}}

\newcommand{\pp}{\mathbb{P}}

\newcommand{\rr}{\mathbb{R}}
\newcommand{\smallavg}[1]{\langle #1 \rangle}

\newcommand{\var}{\mathrm{Var}}

\newcommand{\zz}{\mathbb{Z}}

\newcommand{\fpar}[2]{\frac{\partial #1}{\partial #2}}

\numberwithin{equation}{section}

\newcommand{\eq}[1]{\begin{align*} #1 \end{align*}}
\newcommand{\eeq}[1]{\begin{align} \begin{split} #1 \end{split} \end{align}}
\newcommand{\tv}{d_{\textup{TV}}}

\begin{document}
\title[Lower bounds on fluctuations]{A general method for lower bounds on fluctuations of random variables}
\author{Sourav Chatterjee}
\address{\newline Department of Statistics \newline Stanford University\newline Sequoia Hall, 390 Serra Mall \newline Stanford, CA 94305\newline \newline \textup{\tt souravc@stanford.edu}}
\thanks{Research partially supported by NSF grant DMS-1608249}

\begin{abstract}
There are many ways of establishing upper bounds on fluctuations of random variables, but there is no systematic approach for lower bounds. As a result, lower bounds are unknown in many important problems. This paper introduces a general method for lower bounds on fluctuations.  The method is used to obtain  new results for the stochastic traveling salesman problem, the stochastic minimal matching problem, the random assignment problem, the Sherrington--Kirkpatrick model of spin glasses, first-passage percolation and random matrices. A long list of open problems is provided at the end.
\end{abstract}

\maketitle

\tableofcontents

\vskip.5in
\section{Theory}
\subsection{The problem of lower bounds}\label{intro}
The problem of establishing upper bounds on fluctuations of random variables is widely studied in the literature on concentration inequalities~\cite{blm13, ledoux01}. The theory for lower bounds, however, is not so well-developed. In fact, the  only available methods for computing lower bounds are the following:
\begin{enumerate}
\item Prove a distributional limit theorem. This is possible only in problems where classical tools are applicable, and in a small number of modern problems that admit exact calculations. For most of the contemporary `hard' problems that do not have a miraculous exactly solvable structure, we are very far from proving distributional limit theorems. 
\item Prove a lower bound on some central moment (such as the variance) and a matching upper bound on a higher central moment. The Paley--Zygmund second moment method would then give a lower bound on the order of fluctuations. There is a general method for obtaining lower bounds on variances, due to \citet{wehraizenman90}. However, matching upper bounds are rarely available. First-passage percolation is one example where  the best known upper and lower bounds on fluctuations do not match~\cite{newmanpiza,pemantleperes}. 
\item A coupling technique invented by \citet{janson94} and used by \citet{bollobasjanson} to obtain the first nontrivial lower bound on the fluctuations of the longest increasing subsequence in a random permutation. The main lemma of~\cite{janson94} is closely related to the method of this paper.
\item Problem-specific techniques, as in \cite{gongetal, houdrema, houdrematzinger, lembermatzinger, rhee91}. Some of these are also related to the method proposed here.
\end{enumerate}
Besides the above, there is a recent work of \citet{jansonwarnke} on a general lower bound for lower tails of sums of weakly dependent binary random variables. This bound, however, is for the large deviation regime; it is not meant to be used for understanding typical fluctuations. In the examples where the Janson--Warnke lower bound applies, typical fluctuations can be understood more comprehensively by proving central limit theorems using existing technology for sums of weakly dependent random variables.

What is implicit in the above discussion is that while an upper bound on the variance gives an upper bound on the order of fluctuations by Chebychev's inequality, a lower bound on the variance cannot be used on its own for demonstrating a lower bound on the order of fluctuations. For example, one can easily construct a sequence of random variables which converge in probability to a deterministic constant, but whose variances stay bounded away from zero. In such an example, it is unreasonable to say that the fluctuations do not tend to zero.  In the absence of a simple numerical measure for lower bounds on fluctuations, the following definition looks reasonable.
\begin{defn}\label{lowdef}
Let $\{X_n\}_{n\ge 1}$ be a sequence of random variables and let $\{\delta_n\}_{n\ge 1}$ be a sequence of positive real numbers. We will say that $X_n$ has fluctuations of order at least $\delta_n$ if there are positive constants $c_1$ and $c_2$ such that for all large $n$, and for all $-\infty < a\le b<\infty$ with $b-a\le c_1 \delta_n$, $\pp(a\le X_n\le b) \le 1-c_2$.
\end{defn}
In other words, $X_n$ has fluctuations of order at least $\delta_n$ if it is impossible to construct a sequence of intervals $I_n$, such that $I_n$ has length of order $\delta_n$ and $\pp(X_n\in I_n)\to 1$ as $n\to \infty$. It is easy to see that if $\delta_n^{-1}X_n$ tends to a non-degenerate limit in distribution, then $X_n$ has fluctuations of order at least $\delta_n$ according to the above definition. It is also easy to see that if $X_n$ has fluctuations of order $\delta_n$ according to the above definition, then $\var(X_n)$ is at least of order $\delta_n^2$ (but not vice versa).

Definition \ref{lowdef} is closely related to the notion of concentration functions introduced by~\citet{levy37}. The concentration function $f$ of a random variable $X$ is defined as 
\[
f(l) := \sup_{x\in \rr} \pp(x\le X\le x+l). 
\]
In the language of concentration functions, Definition \ref{lowdef} can be restated as follows: A sequence of random variables $\{X_n\}_{n\ge 1}$, with concentration functions $\{f_n\}_{n\ge 1}$, has fluctuations of order at least $\delta_n$ as $n\to \infty$ if for some $c>0$, 
\[
\limsup_{n\to \infty} f_n(c\delta_n) < 1. 
\]
The theory of L\'evy concentration functions is well-developed for sums of independent random variables~\cite{petrov75} and more generally in situations where a central limit theorem can be proved~\cite{chenetal0}, but not for more complicated objects, especially where distributional limits cannot be established by existing methods.

\subsection{Lower bounds via coupling}
The following simple lemma gives a coupling technique for uniform upper bounds on probabilities of intervals of a given length. The idea is that given a random variable $X$ and a number $\delta$, we construct another random variable $Y$ on the same probability space, such that the law of $X$ and the law of $Y$ are close in total variation distance, and yet there is a substantial chance of $|X-Y|> \delta$. Under these circumstances, the following lemma shows that the probability of $X$ belonging to an interval of length $\le \delta$ is bounded away from one. This is the main tool of this paper. 
\begin{lmm}\label{keylmm}
Let $X$ and $Y$ be two random variables defined on the same probability space. Then for any $-\infty<a\le b<\infty$,
\[
\pp(a\le X\le b) \le \frac{1}{2}(1+\pp(|X-Y|\le b-a)+ \tv(\ml_X, \ml_Y)),
\]
where $\ml_X$ is the law of $X$, $\ml_Y$ is the law of $Y$, and $\tv$ is total variation distance.
\end{lmm}
\begin{proof}
Let $I$ denote the interval $[a,b]$. Then note that 
\eq{
1 &\ge \pp(\{X\in I\}\cup \{Y\in I\})\\
&= \pp(X\in I) + \pp(Y\in I) - \pp(\{X\in I\}\cap \{Y\in I\}).
}
But
\eq{
\pp(Y\in I)&\ge \pp(X\in I)-\tv(\ml_X, \ml_Y),
}
and 
\eq{
\pp(\{X\in I\}\cap \{Y\in I\}) &\le \pp(|X-Y|\le b-a).
}
The proof is completed by combining the above inequalities.
\end{proof}
Some variants of this coupling approach for lower bounds on fluctuations are  already present in \cite{gongetal, houdrema, houdrematzinger, lembermatzinger, rhee91} for the specific problems handled in those papers; but the potential generality of the idea was not recognized in earlier works. \citet{janson94} proved a similar lemma, but where $Y$ was assumed to have the same law as $X$. As far as I know, the exact statement of Lemma \ref{keylmm} is actually a new result. We will see later how this nearly trivial lemma can be used to obtain optimal lower bounds on the orders of fluctuations of some highly complicated random variables. 

Incidentally, one can have a version of Lemma \ref{keylmm} with twice the Kolmogorov distance instead of the total variation distance on the right. The proof makes it clear that this stronger statement is valid. However, bounding the  total variation distance is often more manageable in problems of interest (as we will see in all of the examples in this paper), which is why the lemma is presented in the above form. 

\subsection{A simple example}\label{lemmasec}
For an elementary application of Lemma \ref{keylmm} that uses nothing more than Chebychev's inequality, consider the following example. Let $X_1,\ldots, X_n$ be i.i.d.~Bernoulli$(1/2)$ random variables, and let 
\[
S_n = X_1+\cdots+X_n.
\]
By the central limit theorem, we know that $S_n$ has fluctuations of order $n^{1/2}$. Can we prove a lower bound of order $n^{1/2}$  using Lemma~\ref{keylmm}? To do this, we first define a suitable perturbation $(X_1',\ldots, X_n')$ of the vector $(X_1,\ldots,X_n)$, coupled with $(X_1,\ldots, X_n)$ on the same probability space. Let $\alpha\in (0,1)$ be a constant, to be chosen later, and let $\ep := \alpha n^{-1/2}$. For each $i$, let
\[
X_i' :=
\begin{cases}
X_i &\text{ with probability $1-\ep$,}\\
1 &\text{ with probability $\ep$.}
\end{cases}
\]
Then note that $X_1',\ldots,X_n'$ are i.i.d.~Bernoulli$((1+\ep)/2)$ random variables. Recall that if $\mu$ and $\nu$ are probability measures on a set $\Omega$, and $\nu$ has density $f$ with respect to $\mu$, then 
\[
\tv(\mu, \nu) = \int_\Omega(1-f)_+d\mu,
\]
where $x_+$ denotes the positive part of a real number $x$. Using this representation, a simple calculation shows that 
\eq{
\tv(\ml_{(X_1,\ldots, X_n)}, \ml_{(X_1',\ldots, X_n')}) &= \ee\bigl(1-(1+\ep)^{S_n} (1-\ep)^{n-S_n}\bigr)_+\\
&= \ee\biggl(1-(1-\ep^2)^{n/2}\biggl(\frac{1+\ep}{1-\ep}\biggr)^{S_n - n/2}\biggr)_+. 
}
The quantity within the expectation is always between $0$ and $1$. Moreover, by Chebychev's inequality, for any $\beta>0$ we have
\[
\pp(|S_n - n/2| \ge \beta n^{1/2})\le \frac{1}{4\beta^2}.
\]
Choosing $\beta = \alpha^{-1/2}$, we get
\eq{
&\tv(\ml_{(X_1,\ldots, X_n)}, \ml_{(X_1',\ldots, X_n')}) \\
&\le \frac{\alpha}{4} + 1-\biggl(1-\frac{\alpha^2}{ n}\biggr)^{n/2} \biggl(\frac{1+\alpha n^{-1/2}}{1-\alpha n^{-1/2}}\biggr)^{-\alpha^{-1/2} n^{1/2}}\\
&\le C\sqrt{\alpha},
}
where $C$ does not depend on $n$ or the choice of $\alpha$. Thus, we may choose $\alpha$ so small that the total variation distance is $\le 1/2$ for all $n$.  Next, let 
\[
S_n' := X_1'+\cdots+X_n'.
\]
Notice that for each $i$, $X_i'= X_i$ with probability $1-\ep/2$ and $X_i'=X_i+1$ with probability $\ep/2$. From this observation, it is not difficult to prove using Chebychev's inequality that 
\[
\lim_{n\to \infty}\pp\biggl(S_n' \ge S_n + \frac{\alpha n^{1/2}}{3}\biggr) =1.
\]
Choose $n$ so large that the above probability is at least $2/3$. Since
\[
\tv(\ml_{S_n}, \ml_{S_n'})\le \tv(\ml_{(X_1,\ldots, X_n)}, \ml_{(X_1',\ldots, X_n')})\le \frac{1}{2},
\]
Lemma~\ref{keylmm} now implies that for any interval $I$ of length less than $\alpha n^{1/2}/3$,
\[
\pp(S_n\in I) \le \frac{1}{2}\biggl(1+\frac{1}{2}+ \frac{1}{3}\biggr) = \frac{11}{12}. 
\]
This shows that in the sense of Definition \ref{lowdef}, $S_n$ has fluctuations of order at least $n^{1/2}$. 

\subsection{Total variation distance between product measures}
The above example may appear to be a little ad hoc, because we used an explicit formula for the total variation distance to do our calculations. This, however, is not the case. There is a systematic way of upper bounding the total variation distances between product measures using a measure of similarity between probability measures called the Hellinger affinity. Although this method is well-known and discussed at length in various texts (for example, in~\cite{lecamyang} and \cite{levinetal}), I will now give a quick summary for the sake of completeness and to save the reader the trouble of looking up references. 

Let $(\Omega, \mf)$ be a measurable space, and let $\nu$ be a probability measure on this space. Let $f$ and $g$ be two probability densities with respect to $\nu$, and let $\mu$ and $\mu'$ denote the corresponding measures. The Hellinger affinity or Hellinger integral between $\mu$ and $\mu'$ is defined as
\[
\rho(\mu, \mu') := \int_\Omega \sqrt{fg} \, d\nu. 
\] 
It is easy to see that this does not depend on the choice of $\nu$; the result is the same for any $\nu$ such that $\mu$ and $\mu'$ are both absolutely continuous with respect to $\nu$. For any $\mu$ and $\mu'$, there is always at least one such $\nu$, for example $(\mu+\mu')/2$. The Hellinger affinity  gives the following upper bound on the total variation distance. This is a classical result. The proof is reproduced for completeness. 
\begin{lmm}\label{tvbclmm}
For any two probability measures $\mu$ and $\mu'$ defined on the same measurable space,
\[
\tv(\mu, \mu') \le \sqrt{1-\rho(\mu,\mu')^2}. 
\]
\end{lmm}
\begin{proof}
By the Cauchy--Schwarz inequality,
\eq{
\tv(\mu, \mu') &= \frac{1}{2}\int|f-g|\,d \nu \\
&= \frac{1}{2}\int \bigl|\bigl(\sqrt{f}-\sqrt{g}\bigr)\bigl(\sqrt{f}+\sqrt{g}\bigr)\bigr| \, d\nu\\
&\le \frac{1}{2}\biggl(\int \bigl(\sqrt{f}-\sqrt{g}\bigr)^2 \, d\nu \int \bigl(\sqrt{f}+\sqrt{g}\bigr)^2 \, d\nu \biggr)^{1/2}\\
&= \bigl(1-\rho(\mu, \mu')\bigr)^{1/2}\bigl(1+\rho(\mu, \mu')\bigr)^{1/2},
}
where the last step follows because $\int fd \nu=\int gd \nu=1$.
\end{proof}
The main advantage of using the Hellinger affinity is that it is easy to evaluate for product measures. For $i=1,\ldots, n$, let $\mu_i$ and $\mu_i'$ be probability measures on some measurable space $(\Omega_i, \mf_i)$.  Let $\mu = \mu_1\times \cdots \times \mu_n$ and $\mu' = \mu_1'\times \cdots \times \mu_n'$ be the corresponding product measures on $\Omega_1\times \cdots \times \Omega_n$. Then from the definition of the Hellinger affinity it is clear that
\eq{
\rho(\mu, \mu') &= \prod_{i=1}^n \rho(\mu_i, \mu_i'). 
}
Combining this with Lemma \ref{tvbclmm}, we get the following upper bound for the total variation distance between product measures.
\begin{lmm}\label{tvprod}
Let $\mu$ and $\mu'$ be as in the above paragraph. Then
\[
\tv(\mu, \mu') \le \biggl(1-\prod_{i=1}^n\rho(\mu_i, \mu_i')^2\biggr)^{1/2}. 
\]
\end{lmm}
Let us see what this lemma gives for the example from Section \ref{lemmasec}. In that example, a simple calculation shows that 
\[
\rho(\ml_{X_i}, \ml_{X_i'}) = \frac{\sqrt{1+\ep}+\sqrt{1-\ep}}{2}\ge 1- C\ep^2 = 1-\frac{C\alpha^2}{n},
\]
where $C$ is a positive constant that does not depend on $n$ or $\alpha$.
Thus, 
\[
\tv(\ml_{(X_1,\ldots,X_n)}, \ml_{(X_1',\ldots,X_n')}) \le \biggl(1- \biggl(1-\frac{C\alpha^2}{n}\biggr)^{2n}\biggr)^{1/2}\le C'\alpha,
\]
where $C'$ is another constant that does not depend on $n$ or $\alpha$. This recovers a stronger version of the result that we previously derived using the explicit formula for the total variation distance and Chebychev's inequality. 

\subsection{Perturbative coupling}\label{pertsec}
To apply Lemma \ref{keylmm}, we need to get upper bounds on the total variation distance between the law of a random variable and the law of a small perturbation of the variable. We have dealt with perturbations of independent Bernoulli random variables in the previous section. This section gives a similar bound for continuous random vectors, where the perturbation is of a different nature. The bound is derived using Lemma \ref{tvprod}. We will deal with the following classes of probability measures on $\rr^d$, $d\ge 1$.
\begin{defn}
Let $\cp(d)$ denote the set of all probability densities of the form $e^{-V}$ on $\rr^d$, where  $V$ is a $C^\infty$ function, $V$ and all its derivatives have at most polynomial growth at infinity, and $e^{V}$ increases faster than any polynomial at infinity. 
\end{defn}
\begin{defn}
Let $\cp^+(d)$ denote the set of all probability densities of the form $e^{-V}$  on $[0,\infty)^d$, such that $V$ is a $C^\infty$ function in $(0,\infty)^d$, $V$ and all its derivatives have at most polynomial growth at infinity and extend continuously to the boundary of $[0,\infty)^d$, and $e^{V}$ increases faster than any polynomial at infinity.
\end{defn} 
Many of the familiar probability distributions, such as Gaussian distributions and exponential distributions, are of the above type. Some other commonly used distributions, such as uniform distributions on compact sets, do not belong to the above classes. 

The following theorem gives a lower bound on the Hellinger affinity between a probability measure in $\cp(d)$  or $\cp^+(d)$ and a certain kind of perturbation of that measure. The proof is based on an application of the divergence theorem. This result is frequently used in later sections.
\begin{thm}\label{multthm}
Let $X$ be a $d$-dimensional random vector whose law belongs to either $\cp(d)$ or $\cp^+(d)$. Take any $\ep \in (-1/2,1/2)$, and let $X' = X/(1+\ep)$. Then
\[
\rho(\ml_X, \ml_{X'}) \ge 1-C\ep^2
\]
where $C$ depends only on $\ml_X$ and $d$.
\end{thm}
\begin{proof}
Let $e^{-V}$ be the probability density function of $X$. In the following, all integrals are over $\rr^d$ if $e^{-V}\in \cp(d)$, and over $[0,\infty)^d$ if $e^{-V}\in \cp^+(d)$. Let
\eq{
g(\ep) := \int \sqrt{(1+\ep)^de^{-V(x+\ep x)}e^{-V(x)}} \,d x = \rho(\ml_X, \ml_{X'}).
}
By the inequality 
\[
|e^x-e^y|\le \frac{1}{2}|x-y|(e^x+e^y), 
\]
we have that for any $\ep_1$ and $\ep_2$,
\eq{
&|e^{-V(x+\ep_1 x)} - e^{-V(x+\ep_2 x)}|\\
&\le \frac{1}{2}|V(x+\ep_1x)-V(x+\ep_2x)|(e^{-V(x+\ep_1 x)} + e^{-V(x+\ep_2 x)}).
}
Using this inequality and the assumptions on $V$, it is not difficult to verify that $g$ is a $C^\infty$ function on $(-1/2,1/2)$, and the derivatives can be computed by differentiating under the integral. Note that $g(0)=1$ and 
\eq{
g'(\ep) &= \frac{d}{2}(1+\ep)^{-1}g(\ep) \\
&\qquad - \frac{1}{2}(1+\ep)^{d/2}\int x\cdot \nabla V(x+\ep x) e^{-(V(x+\ep x)+V(x))/2}\, d x,
}
which gives
\eq{
g'(0)&=\frac{d}{2} - \frac{1}{2}\int x\cdot \nabla V(x) e^{-V(x)} \, d x\\
&=  \frac{1}{2}\int (d-x\cdot \nabla V(x)) e^{-V(x)} \, d x.
}
Let $h = (h_1,\ldots, h_d)$ be the function defined as
\eq{
h_i(x) = x_i e^{-V(x)}.
}
Then 
\eq{
\fpar{h_i}{x_i} = \biggl(1-x_i \fpar{V}{x_i}\biggr) e^{-V(x)},
}
and hence
\eq{
\mathrm{div}\, h(x)= \sum_{i=1}^d\fpar{h_i}{x_i} = (d-x\cdot \nabla V(x)) e^{-V(x)}. 
}
If $e^{-V}\in \cp(d)$, then using the growth assumptions about $V$, a simple application of the divergence theorem now shows that $g'(0)=0$. If $e^{-V}\in \cp^+(d)$, then also the divergence theorem implies that $g'(0)=0$, because $h(x)\cdot n(x)=0$ on the boundary of $[0,\infty)^d$, where $n(x)$ denotes the unit normal vector to the boundary at the point $x$.  Finally, note that
\eq{
g''(\ep) &= -\frac{d}{2}(1+\ep)^{-2} g(\ep) + \frac{d}{2}(1+\ep)^{-1} g'(\ep) \\
&\qquad - \frac{d}{4}(1+\ep)^{(d-2)/2}\int x\cdot \nabla V(x+\ep x) e^{-(V(x+\ep x)+V(x))/2}\, d x \\
&\qquad - \frac{1}{2}(1+\ep)^{d/2}\int \biggl(x\cdot \mathrm{Hess}\, V(x+\ep x)\, x\\
&\qquad \qquad \qquad\qquad  \qquad - \frac{1}{2}(x\cdot \nabla V(x+\ep x))^2\biggr) e^{-(V(x+\ep x)+V(x))/2}\, d x,
}
where $ \mathrm{Hess}\, V$ is the Hessian matrix of $V$. 
By the assumed properties of $V$, the above formulas show that $|g|$, $|g'|$ and $|g''|$ are uniformly bounded in the interval $(-1/2, 1/2)$. Moreover, we have already deduced that $g(0)=1$ and $g'(0)=0$. Thus, for any $\ep\in (-1/2, 1/2)$, 
\eq{
g(\ep)\ge 1-C\ep^2,
}
where $C$ depends only on $V$ and $d$.
\end{proof} 
Combining this theorem with Lemma \ref{tvprod} yields the following corollary.
\begin{cor}\label{multcor}
Let $X_1,\ldots, X_n$ be i.i.d.~$d$-dimensional random vectors with probability density belonging to either $\cp(d)$ or $\cp^+(d)$. Take any $\ep_1,\ldots,\ep_n\in (-1/2,1/2)$ and let $X_i'=X_i/(1+\ep_i)$ for $i=1,\ldots, n$. Then
\eq{
\tv(\ml_{(X_1,\ldots, X_n)}, \ml_{(X_1',\ldots, X_n')})\le C\biggl(\sum_{i=1}^n \ep_i^2\biggr)^{1/2}.
}
where $C$ depends only of the law of the $X_i$'s and the dimension $d$.
\end{cor}
\begin{proof}
By Theorem \ref{multthm} and Lemma \ref{tvprod}, 
\[
\tv(\ml_{(X_1,\ldots, X_n)}, \ml_{(X_1',\ldots, X_n')}) \le \biggl(1-\prod_{i=1}^n (1-C_0\ep_i^2)^2\biggr)^{1/2}, 
\]
where $C_0$ depends only on the law of the $X_i$'s and the dimension $d$. Since the total variation distance is bounded by one, we may assume without loss of generality (by suitably increasing the value of $C$ in the statement of the corollary, if necessary) that 
\[
\sum_{i=1}^n \ep_i^2 \le \frac{1}{4C_0}.
\] 
Then by repeated applications of the inequality $(1-x)(1-y)\ge 1-x-y$, which holds for $x,y\in [0,1]$, we get the desired result.
\end{proof}
\subsection{Connection with the Mermin--Wagner theorem}
The technique of using Lemma \ref{keylmm} in conjunction with Lemma \ref{tvprod} and Theorem \ref{multthm} has strong similarities with the celebrated Mermin--Wagner theorem of statistical physics \cite{mermin, merminwagner}. Roughly speaking, the Mermin--Wagner theorem shows that continuous symmetries cannot be spontaneously broken in dimensions $\le 2$. The physics proof of the Mermin--Wagner theorem involves the introduction of a slowly rotating perturbation known as a `spin wave', akin to the perturbation used in Theorem \ref{multthm}. The first rigorous proof of the Mermin--Wagner theorem, due to \citet{mcbryanspencer}, used complex analytic techniques. A later proof, due to \citet{pfister}, used a more transparent argument that resembles the  method of this paper in various aspects. For a modern exposition in probabilistic language, see the lecture notes of \citet{peledspinka}.

\section{Applications}
\subsection{Traveling salesman and minimal matching}
Let $f$ be a measurable real-valued function on $(\rr^d)^n$ such that there is some $r>0$ so that for any $\lambda \ge 0$ and any $x_1,\ldots, x_n\in\rr^d$,
\eeq{
f(\lambda x_1,\ldots, \lambda x_n) = \lambda^r f(x_1,\ldots, x_n).\label{homog}
}
Such functions arise in many geometric combinatorial optimization problems. The length of the optimal traveling salesman path through $x_1,\ldots, x_n$, the length of the minimal spanning tree and the length of the minimal matching (if  $n$ is even) are all examples of functions having the above property, with $r=1$. The volume of the convex hull of $x_1,\ldots, x_n$ is an example that satisfies \eqref{homog} with $r=d$. 

If the points $x_1,\ldots, x_n$ are replaced by i.i.d.~random points $X_1,\ldots, X_n$ drawn from some probability measure on $\rr^d$, the resulting problem is called the stochastic version of the original optimization problem. The stochastic problems have been extensively studied by probabilists. Laws of large numbers and concentration inequalities for upper bounds on fluctuations are well understood due to the works of many authors, and beautifully exposited in the classic monograph of~\citet{steele97}. Yet, a lot remains to be understood. For example, the distribution theories for the stochastic traveling salesman and the stochastic minimal matching problems remain out of the reach of available technology.

Not much is known about lower bounds on fluctuations in the problems where the distribution theory is not understood. For the stochastic traveling salesman problem, the only result on lower bounds that I am aware of is a result of \citet{rhee91}, who proved a fluctuation lower bound of the correct order in dimension two when the points are uniformly distributed in the unit square. Nothing is known about lower bounds in the stochastic minimal matching problem. The following theorem gives a general lower bound for fluctuations when the points are distributed according to some probability density in  $\cp(d)$ or $\cp^+(d)$.
\begin{thm}\label{optthm}
Take any $d\ge 1$. Let $X_1,X_2,\ldots$ be i.i.d.~random vectors with probability density in either $\cp(d)$ or $\cp^+(d)$. For each $n$, let $f_n:(\rr^d)^n\to \rr$ be a function satisfying \eqref{homog} for some fixed $r>0$, and let $L_n := f_n(X_1,\ldots, X_n)$. Let $(t_n)_{n\ge 1}$ be a sequence of positive real numbers such that
\[
\liminf_{n\to \infty}\pp(L_n\ge t_n) >0.
\]
Then $L_n$ has fluctuations of order at least $n^{-1/2} t_n$, in the sense of Definition~\ref{lowdef}. 
\end{thm}
\begin{proof}
Without loss of generality, let $a$ be a positive constant such that 
\[
\pp(L_n \ge t_n)\ge a
\]
for all $n$. Take any $n$. For $i=1,\ldots, n$, let
\[
X_i' := \frac{X_i}{1+\alpha n^{-1/2}},
\]
where $\alpha$ will be determined later. Let $L_n' := f_n(X_1',\ldots, X_n')$. Then by Corollary \ref{multcor}, 
\eq{
\tv(\ml_{L_n}, \ml_{L_n'})&\le \tv(\ml_{(X_1,\ldots,X_n)}, \ml_{(X_1',\ldots, X_n')}) \le C\alpha,
} 
where $C$ depends only on $\ml_{X_1}$ and $d$. On the other hand, by \eqref{homog},
\[
L_n' = \frac{L_n}{(1+\alpha n^{-1/2})^r}. 
\]
Consequently, for any $\beta$ such that
\[
0<\beta < \frac{(1+\alpha n^{-1/2})^r - 1}{n^{-1/2}(1+\alpha n^{-1/2})^r} = r\alpha + o(1)
\]
(where $o(1)$ denotes a quantity that tends to zero as $n\to \infty$), we have
\eq{
\pp(L_n'\ge L_n - \beta n^{-1/2}t_n) &= \pp\biggl(L_n \le \frac{(1+\alpha n^{-1/2})^r\beta n^{-1/2} t_n}{(1+\alpha n^{-1/2})^r - 1}\biggr)\\
&\le \pp(L_n < t_n) \le 1-a.
}
Thus, choosing $\alpha$ sufficiently small, and then choosing $\beta$  depending on $\alpha$, Lemma~\ref{keylmm} completes the proof. 
\end{proof}
Concretely, Theorem \ref{optthm} gives the following lower bound in the traveling salesman and minimal matching problems.
\begin{cor}
Take any $d\ge 2$. In the stochastic traveling salesman and stochastic minimal matching problems with points drawn according to some probability density in $\cp(d)$ or $\cp^+(d)$, the optimal lengths have fluctuations of order at least $n^{(d-2)/2d}$, in the sense of Definition \ref{lowdef}.
\end{cor}
\begin{proof}
In the stochastic traveling salesman problem and the stochastic minimal matching problems in $d\ge 2$, it is not difficult to see that the optimal values are bounded below by some constant multiple of the sum of nearest-neighbor distances. From this, it follows by a simple mean and variance calculation that $t_n$ can be chosen to be $n^{1-1/d}$, since there are $n$ points and the nearest neighbor distances are of order at least $n^{-1/d}$ (because a density in $\cp(d)$ or $\cp^+(d)$ is necessarily bounded). With this choice of $t_n$, Theorem \ref{optthm} implies that the order of fluctuations is at least $n^{-1/2} n^{1-1/d} = n^{(d-2)/2d}$. 
\end{proof}
Interestingly, the lower bound in the above corollary matches the known order of upper bounds on fluctuations in these problems~\cite{steele97}. The known upper bounds, however, are for points drawn from the uniform distribution on $[0,1]^d$. I do not know if upper bounds are known for unbounded distributions, such as the ones in~$\cp(d)$ and $\cp^+(d)$.

The next theorem gives the optimal lower bound for the traveling salesman problem when the points are distributed uniformly in $[0,1]^d$. This is a straightforward generalization of the lower bound obtained by \citet{rhee91} in $d=2$. The coupling used in the proof of this theorem is borrowed from Rhee's paper. I do not know if a simpler coupling can be made to work.
\begin{thm}
Take any $d\ge 2$. Let $L_n$ be the length of the optimal tour in the stochastic traveling salesman problem with $n$ points when the points are distributed independently and uniformly in $[0,1]^d$. Then $L_n$ has fluctuations of order at least $n^{(d-2)/2d}$, in the sense of Definition \ref{lowdef}.
\end{thm}
\begin{proof}
Throughout this proof, $C$ will denote any positive universal constant, whose value may change from line to line.

Let $X_1,\ldots, X_n$ be i.i.d.~uniform points from $[0,1]^d$. Assume that $n\ge 4$ and let $m = [n/2]$. Given $X_1,\ldots, X_m$, let $D$ be the set of all points in $[0,1]^d$ that are within distance $\alpha n^{-1/d}$ from the set $\{X_1,\ldots, X_m\}$, where $\alpha\in (0,1)$ will be chosen later. Generate $Y_{m+1}, \ldots, Y_n$ independently and uniformly from the set $D$. For each $m+1\le i\le n$, let 
\[
X_i' :=
\begin{cases}
X_i &\text{ with probability $1-\beta n^{-1/2}$,}\\
Y_i &\text{ with probability $\beta n^{-1/2}$,}
\end{cases}
\]
where $\beta\in (0,1)$ will be chosen later. For $1\le i\le m$, let $X_i' := X_i$. Let $L_n$ be the length of the optimal tour through $X_1,\ldots, X_n$ and let $L_n'$ be the length of the optimal tour through $X_1',\ldots, X_n'$. 

Given $X_1,\ldots, X_m$, the random variables $X_{m+1}',\ldots, X_n'$ are i.i.d.~with probability density function
\eq{
1-\beta n^{-1/2} + \frac{\beta n^{-1/2} }{\vol(D)} 1_{\{x\in D\}}
}
at $x\in [0,1]^d$. From this formula and the inequality $\sqrt{1+a}\ge 1+a/2 - Ca^2$ that holds for $a\ge -1/2$, an easy calculation shows that the Hellinger affinity between this conditional law and the uniform distribution on $[0,1]^d$ is bounded below by
\eq{
1-\frac{C\beta^2}{\vol(D)n}.
}
Therefore by Lemma \ref{tvprod}, the total variation distance between the conditional law of $(X_{m+1}',\ldots, X_n')$ and that of $(X_{m+1},\ldots, X_n)$ is bounded above~by 
\[
\frac{C\beta}{\sqrt{\vol(D)}}. 
\]
It is not difficult to show (for example, as in the proof of \cite[Lemma 7]{rhee91}) that with probability tending to one as $n\to \infty$, $\vol(D)\ge C\alpha^d$. Combining this with the above bound on the conditional laws, it follows that
\eeq{
\tv(\ml_{(X_1,\ldots, X_n)}, \ml_{(X_1',\ldots, X_n')}) \le \frac{C\beta}{\alpha^{d/2}}. \label{tsptv}
}
Consider the optimal tour through $X_1,\ldots, X_n$. Each  $X_i$ has two `neighbors' in this tour, one which comes before it and one that comes after. Call these points $U_i$ and $V_i$. If the point $X_i$ is erased, then the length of the optimal tour must decrease by at least 
\[
\|X_i-U_i\|+\|X_i-V_i\| - \|U_i-V_i\|. 
\]
Let $K$ be a positive real number, to be chosen later. It is known that with probability tending to one, the optimal tour has length $\le Cn^{1-1/d}$ \cite[Chapter 2]{steele97}. If this happens, then the average length of an edge in the tour is bounded above by $Cn^{-1/d}$, and hence the fraction of all $i$ for which 
\eeq{
\max\{\|X_i-U_i\|,\, \|X_i-V_i\|\} \le Kn^{-1/d}\label{xuv}
}
is bounded below by $1-CK^{-1}$. 

For each $i$, let $N_i$ be the set of all $j\ne i$ such that $\|X_i-X_j\|\le Kn^{-1/d}$. Let
\[
R_i := \min\{\|X_i-X_j\|+\|X_i-X_k\|-\|X_j-X_k\|: j,k\in N_i\},
\]
where we follow the usual convention that the minimum of an empty set is infinity. Let $\gamma$ be a positive real number, to be chosen later. From the local structure of a set of uniformly distributed points, it is not difficult to prove that with probability tending to one as $n\to \infty$, the fraction of all $i$ for which $R_i\ge \gamma n^{-1/d}$ is at least $c(\gamma,K)$, where $c(\gamma,K)\to 1$ as $\gamma \to 0$ for any fixed $K$. 

Finally, note that if the point $X_i$ is dropped, and \eqref{xuv} holds, then the length of the optimal tour decreases by at least $R_i$. 

Combining all of the above observations, we see that if $K$ is chosen large enough, and then $\gamma$ is chosen small enough depending on $K$, then with probability tending to one as $n\to \infty$, there are at least $3n/4$ points $X_i$ such that \eqref{xuv} holds and $R_i\ge \gamma n^{-1/d}$. Let $A$ be the set of all $i$ such that $X_i$ has these two properties.

Let $B$ be the set of all $m+1\le i\le n$ such that $X_i'=Y_i$. Then from the conclusion of the previous paragraph, it follows that with probability tending to one, $|A\cap B|\ge C\beta n^{1/2}$. Also, the expected number of pairs of points in $B$ that are neighbors of each other in the optimal tour is of order $1$ as $n\to \infty$. To see this, just note that $B$ is a randomly chosen subset of size $O(\sqrt{n})$ and compute a straightforward bound on the conditional expectation of the number of such pairs given the $X_i$'s. From these two observations, it follows that with probability tending to one, deleting $\{X_i:i\in B\}$ results in a decrease of at least $C\beta \gamma n^{1/2-1/d}$ in the length of the optimal tour. 

After dropping $\{X_i:i\in B\}$, let us now replace these points by $\{Y_i:i\in B\}$. Note that the resulting point set is exactly $\{X_1',\ldots, X_n'\}$.  Each $Y_i$ is within distance $\alpha n^{-1/d}$ of some $X_j$, and therefore adds at most $2\alpha n^{-1/d}$ to the length of the optimal tour. Consequently, with probability tending to one, the total increase after inserting all the $Y_i$'s is at most $C\alpha \beta n^{1/2-1/d}$. Thus, choosing $\alpha$ sufficiently small (depending only on $\gamma$), we can ensure that with probability tending to one, 
\[
L_n-L_n'\ge C\beta n^{1/2-1/d}.
\] 
Finally, choose $\beta$ small enough (depending on $\alpha$) so that the right side of~\eqref{tsptv} is less than $1/2$. Lemma \ref{keylmm} now completes the proof.  
\end{proof}

\subsection{Free energy of the Sherrington--Kirkpatrick model}
Let $n$ be any positive integer. Let $(g_{ij})_{1\le i<j\le n}$ be i.i.d.~standard Gaussian random variables. The Sherrington--Kirkpatrick (S-K)  model of spin glasses~\cite{sherringtonkirkpatrick} with $n$ spins at inverse temperature $\beta\ge 0$ and external field $h\in \rr$ defines a random probability measure on $\{-1,1\}^n$, which puts mass proportional to $e^{\beta H_n(\sigma)}$ at each point $\sigma= (\sigma_1,\ldots, \sigma_n)\in \{-1,1\}^n$, where 
\[
H_n(\sigma) := \frac{1}{\sqrt{n}}\sum_{1\le i<j\le n} g_{ij}\sigma_{i} \sigma_{j} + h \sum_{i=1}^n \sigma_i.
\]
The S-K model has inspired a large body of work in probability theory, extensively surveyed in~\cite{talagrandbook1, talagrandbook2, talagrandbook3, panchenkobook}. One of the key quantities of interest is the free energy of the model, defined as
\eq{
F_n(\beta,h) := \log \sum_{\sigma \in \{-1,1\}^n} e^{\beta H_n(\sigma)}. 
}
For $\beta<1$ and $h=0$, the fluctuations of the free energy are well understood due to the work of \citet{alr}. In this case, the free energy has fluctuations of order $1$ as $n\to \infty$, and satisfies a central limit theorem after centering. When $\beta > 1$ and $h=0$, the best known upper bound on the order of fluctuations is $\sqrt{n/\log n}$~\cite{chatterjee09, chatterjeebook}. I do not know of a definite conjecture about the true order of fluctuations. For $h\ne 0$, \citet{chenetal} have recently proved a central limit theorem for $F_n(\beta,h)$ for any $\beta$, showing that it has fluctuations of order $\sqrt{n}$. When $h=0$, I have heard it said that the fluctuations may be of order $1$, but I have also heard it said that the fluctuations are of order $n^\rho$ for some small $\rho$. The following result shows that in the absence of an external field, the fluctuations are at least of order~$1$.
\begin{thm}\label{skthm}
The free energy of the S-K model at zero external field and any inverse temperature $\beta$, has fluctuations of order at least $1$, in the sense of Definition \ref{lowdef}. 
\end{thm}
\begin{proof}
Since $h=0$, we will write $F_n(\beta)$ instead of $F_n(\beta, h)$. Let $\alpha$ be a positive constant, to be determined later. Let
\[
\tilde{g}_{ij} := \frac{g_{ij}}{1-\alpha n^{-1}},
\]
and let $\tilde{F}_n(\beta)$ be the free energy of the model where $g_{ij}$ is replaced by $\tilde{g}_{ij}$. Then by Corollary \ref{multcor},
\eq{
\tv(\ml_{F_n(\beta)}, \ml_{\tilde{F}_n(\beta)}) \le \tv(\ml_{(g_{ij})_{1\le i<j\le n}}, \ml_{(\tilde{g}_{ij})_{1\le i<j\le n}})\le C\alpha,
}
where $C$ does not depend on $n$. 
A simple computation shows that
\eq{
\tilde{F}_n(\beta) - F_n(\beta) &= \log \bigavg{\exp\biggl(\frac{\beta \alpha H_n(\sigma)}{n(1-\alpha n^{-1})}\biggr)}_\beta,
}
where $\smallavg{\cdot}_\beta$ denotes expectation under the probability measure defined by the S-K model at inverse temperature $\beta$. By Jensen's inequality, this gives
\eeq{
\tilde{F}_n(\beta) - F_n(\beta) &\ge \bigavg{\frac{\beta \alpha H_n(\sigma)}{n(1-\alpha n^{-1})}}_\beta.\label{fnfn}
}
On the other hand, another simple computation gives
\eeq{
F_n'(\beta) &= \smallavg{H_n(\sigma)}_\beta,\label{fnp}
}
where $F_n'$ is the derivative of $F_n$. Now, the following facts are well-known (see, for example, in~\cite{talagrandbook1}):  $n^{-1} F_n(\beta)$ is a convex function of $\beta$, and converges to a deterministic limit $P(\beta)$ as $n\to \infty$. Therefore by the properties of convex functions, $n^{-1}F_n'(\beta)\to P'(\beta)$ for every $\beta>0$ where $P$ is differentiable.  Moreover, it is also known that $P'(\beta)>0$ for every $\beta>0$ where $P$ is differentiable.  Thus, at every $\beta>0$ where $P$ is differentiable, $n^{-1}F_n'(\beta)$ converges to a positive limit.

Now, $P$ is convex and hence almost everywhere differentiable in $(0,\infty)$. In particular, for any $\beta>0$ there exists $\beta'\in (0,\beta)$ where $P$ is differentiable. Thus, $n^{-1}F_n'(\beta')$ converges to a positive limit. But the convexity of $F_n$ implies that $F_n'(\beta')\le F_n'(\beta)$. Thus, there exists a positive constant $a$ (depending on $\beta$) such that 
\[
\lim_{n\to \infty}\pp(n^{-1}F_n'(\beta) \ge a)=1.
\]
Thus, by \eqref{fnfn} and \eqref{fnp}, we see that there is some $b>0$ such that 
\[
\lim_{n\to \infty} \pp(\tilde{F}_n(\beta)  - F_n(\beta) \ge b) = 1.
\]
Choosing $\alpha$ small enough, Lemma~\ref{keylmm} completes the proof. 
\end{proof}
Another important quantity related to the S-K model is its ground state energy, namely,
\eeq{
\max_{\sigma\in \{-1,1\}^n} \frac{1}{\sqrt{n}}\sum_{1\le i<j\le n} g_{ij}\sigma_i \sigma_j.\label{gstate}
}
The best known upper bound on the order of fluctuations of the ground state energy is $o(\sqrt{n})$, proved in a recent manuscript of \citet{chenetal1}. No lower bound is known. It is believed that the correct order of fluctuations is $n^\rho$, where $\rho$ is either $1/6$ or $1/4$ (see, for example, the discussion in~\cite{palassini08}). 
The following theorem proves a lower bound of order $1$ on the fluctuations of the ground state energy. 
\begin{thm}
The ground state energy of the S-K model, as defined in~\eqref{gstate}, has fluctuations of order at least $1$, in the sense of Definition \ref{lowdef}. 
\end{thm}
\begin{proof}
Let $\tilde{g}_{ij}$ be as in the proof of Theorem~\ref{skthm}. Let $G_n$ and $\tilde{G}_n$ be the ground state energies in the two systems. Then 
\eq{
\tilde{G}_n &= \frac{G_n}{1-\alpha n^{-1}}. 
}
Using the well-known fact that $G_n/n$ converges in probability to a deterministic positive limit as $n\to \infty$, it is now easy to complete the proof using Corollary \ref{multcor} and Lemma \ref{keylmm}.
\end{proof}

\subsection{First-passage percolation}
Take any $d\ge 2$. Let $E$ be the set of nearest-neighbor edges of $\zz^d$, and let $(\omega_e)_{e\in E}$ be a collection of i.i.d.~nonnegative random variables, called `edge weights'. Define the weight of a path in $\zz^d$ to be the sum of the weights of the edges in the path. Define the first-passage time $T(x,y)$ from a point $x$ to a point $y$ to be the minimum over the weights of all paths from $x$ to $y$. First-passage percolation is the study of the behavior of these first-passage times. For a recent survey of the large mathematical literature on this model, see~\cite{fppsurvey}. 

A lot of energy has been spent on the study of fluctuations of first-passage times, but the known results are far from optimal. If $x$ and $y$ are points at distance $n$ from each other, the best known upper bound on the fluctuations of $T(x,y)$ is of order $\sqrt{n/\log n}$ under mild assumptions on the law of the edge weights~\cite{bks, br08, damronetal}, although it is conjectured that at least in $d=2$, the correct order should be $n^{1/3}$. The situation with lower bounds is even worse. The best known lower bound on the order of fluctuations in $d=2$, due to~\citet{pemantleperes}, is of order $\sqrt{\log n}$ when the edge weight distribution is exponential. The proof depends crucially on the memoryless property of the exponential distribution. \citet{newmanpiza} showed that for a fairly general class of edge weight distributions, the variance of $T(x,y)$ is lower bounded by a constant multiple of $\log n$. The Newman--Piza lower bound is based on a technique pioneered by \citet{wehraizenman90}. However, on its own, this lower bound on the variance does not give a lower bound on the order of fluctuations since there is no matching upper bound on any higher moment (or even on the variance itself). The following theorem proves the $\sqrt{\log n}$ lower bound on the order of fluctuations for a large class of edge weight distributions.
\begin{thm}
Consider the first-passage percolation model with i.i.d.~nonnegative edge weights in $d=2$. Suppose that the edge weight distribution belongs to the class $\cp^+(1)$. Let $x_n$ and $y_n$ be two sequences of points such that the distance between $x_n$ and $y_n$ grows like a constant multiple of $n$. Then the fluctuations of $T(x_n,y_n)$ are at least of order $\sqrt{\log n}$ in the sense of Definition~\ref{lowdef}. 
\end{thm}
\begin{proof}
Throughout this proof, $C$ will denote any positive constant that may depend only on the edge weight distribution, and nothing else. The value of $C$ may change from line to line or even within a line.

Let $n$ be a positive integer greater than four, and let $y$ be a point at distance $n$ from $0$. Let $T= T(0,y)$. It suffices to prove that $T$ has fluctuations of order at least $\sqrt{\log n}$. 

For each edge $e$, let $k(e)$ denote the distance of $e$ from the origin, where `distance' means the graph distance between the origin and the endpoint of $e$ that is closer to the origin. For all edges $e$ with $k(e)\le n/2$, let 
\eq{
\ep_e := \frac{\alpha}{(k(e)+1)\sqrt{\log n}},
}
where the constant $\alpha$ will be chosen later. Since there are $\le C r$ edges at distance $r$, 
\eeq{
\sum_e \ep_e^2 &\le \sum_{r\le n/2} Cr\frac{\alpha^2}{(r+1)^2\log n}\le C\alpha^2.\label{ep2}
}
Let $T'$ be the first-passage time from $0$ to $y$ when the original edge weights $\omega_e$ are replaced by $\omega_e/(1+\ep_e)$ for all edges with $k(e)\le n/2$. By Corollary~\ref{multcor} and the inequality~\eqref{ep2},
\eeq{
\tv(\ml_T, \ml_{T'}) \le C\alpha. \label{fpptv}
}
Let $m = [n/2]$. Note that if $e$ and $e'$ are two edges that share a vertex, then $|k(e)-k(e')|\le 1$.  Thus,  considering the first $m$ edges $e_1,\ldots, e_m$ in the optimal path, we see that 
\eeq{
T-T' \ge \sum_{i=1}^m \frac{\ep_{e_i} \omega_{e_i}}{1+\ep_{e_i}} \ge \frac{C\alpha}{\sqrt{\log n}}\sum_{i=1}^m \frac{\omega_{e_i}}{i}.\label{ttq}
}
Notice that the optimal path must necessarily be a self-avoiding path. Let $P(x,r)$ be the set of all self-avoiding paths of length $r$ starting at a vertex~$x$. 
For $\theta \ge 0$, let
\eq{
\phi(\theta) := \ee(e^{-\theta \omega_e}),
}
where $e$ denotes a generic edge. Since the density of the edge weight distribution is uniformly bounded,
\eq{
\phi(\theta) &\le C\int_0^\infty e^{-\theta y} \, d y= \frac{C}{\theta}.
}
Take any $x$ and $r$, and any path in $P(x,r)$. Let $e_1,\ldots, e_r$ be the sequence of edges in the path. Then for any $\theta>0$ and $b>0$, 
\eq{
\pp\biggl(\sum_{i=1}^r \omega_{e_i}\le br\biggr) &\le e^{\theta br} \phi(\theta)^r\le C^r \frac{e^{\theta br}}{\theta^r}.
}
Choosing $\theta = 1/b$, we get
\[
\pp\biggl(\sum_{i=1}^r \omega_{e_i}\le br\biggr) \le (Cb)^r.
\]
Fix some $b$. Let $E_r$ be the event that there is a self-avoiding path, starting at some point at distance $\le r$ from the origin, of length $r$ and weight $\le br$. Since there are $\le Cr^2$ points at distance $\le r$ from the origin, and $\le C^r$ paths of length $r$ from any given starting point, the above inequality shows that 
\eeq{
\pp(E_r) \le C^rr^2 (Cb)^r\le C^r b^r.\label{prineq}
}
Let $E_r^c$ denote the complement of $E_r$, and define
\eq{
F := \bigcap_{k=1}^{[\log_2 m]} E_{2^k}^c.
}
Then by the previous \eqref{prineq},
\eeq{
\pp(F) &\ge 1 - \sum_{k=1}^{[\log_2 m]}C^{2^k} b^{2^k}.\label{prineq2}
}
Suppose that $F$ happens.  Then for any path of length $m$, starting at $0$, with edges $e_1,\ldots, e_m$,
\begin{align}
\sum_{i=1}^m \frac{\omega_{e_i}}{i} &\ge \sum_{k=1}^{[\log_2 m]} \sum_{i=2^{k-1}}^{2^{k}-1} \frac{\omega_{e_i}}{i}\nonumber\\
&\ge \sum_{k=1}^{[\log_2 m]}\frac{1}{2^k} \sum_{i=2^{k-1}}^{2^{k}-1} \omega_{e_i}\nonumber \\
&>  \sum_{k=1}^{[\log_2 m]} \frac{b2^{k-1}}{2^k}= \frac{1}{2}b[\log_2 m].\label{wbineq}
\end{align}
By \eqref{ttq}, \eqref{prineq2} and \eqref{wbineq}, choosing $\alpha$ and $b$ small enough, it follows that there exists a positive constant $a$, depending only on the edge weight distribution, such that
\eq{
\pp\bigl(T-T'\ge a\sqrt{\log n}\bigr) \ge \frac{1}{2}.
}
By \eqref{fpptv} and the above inequality, Lemma~\ref{keylmm} completes the proof.
\end{proof}

For $x\in \zz^d$, let $T_x := T(0,x)$. Extend the definition of $T_x$ to $x\in \rr^d$ in some reasonable way such that the map $x\mapsto T_x$ is continuous. For example, one can define $T_x$ in the interior of every $k$-cell of $\zz^d$ as the unique harmonic function that takes specified values on the boundary of the cell, starting with $k=1$ (because $T_x$ is a priori defined on every $0$-cell), and inductively going up to $k=d$. The reason for insisting on the continuity of the extension is that it will make it technically easier for us to carry out a certain step in the proof of the next theorem.

For $t\ge 0$, define the random set
\[
B(t) := \{x\in \rr^d: T_x\le t\}. 
\]
The shape theorem of \citet{coxdurrett} says that under mild conditions on the edge weight distribution, there exists a deterministic compact symmetric convex set $B_0$ with nonempty interior, such that almost surely, for all $\ep>0$, 
\eq{
(1-\ep)B_0 \subseteq \frac{1}{t}B(t) \subseteq (1+\ep)B_0 \ \ \ \text{for all large $t$.}
}
The set $B_0$ is called the `limit shape' of first-passage percolation with the given edge weight distribution. 

A unit vector $x\in \rr^d$ is called a `direction of curvature' if in a neighborhood of the boundary point of $B_0$ in the direction $x$, $B_0$ is `at least as curved as an Euclidean sphere'. Formally, this can be defined as in \cite{newmanpiza}, as follows. Take any unit vector $x$. Let $z$ be the boundary point of $B_0$ in the direction $x$. We say that $x$ is a direction of curvature  if there is an Euclidean ball $D$, with any center, such that $D\supseteq B_0$ and $z\in \partial D$. Simple geometric considerations imply that $B_0$ has at least one direction of curvature. 

Suppose that $d=2$ and $x$ is a direction of curvature. Then, under mild conditions on the edge weight distribution, \citet{newmanpiza} showed that for all $\ep>0$, $\var(T_{nx}) \ge Cn^{1/4-\ep}$ for all $n$, where $C$ does not depend on $n$ (but may depend on the edge weight distribution and $\ep$). However, as we have observed before, this does not actually prove anything about the true order of fluctuations of $T_{nx}$ since we do not have a matching upper bound. The following theorem fills this gap.
\begin{thm}\label{fppcurv}
Suppose that $d=2$ and the edge weight distribution belongs to the class $\cp^+(1)$. Take any $\ep>0$. If $x$ is a direction of curvature, then the first-passage time $T_{nx}$ has fluctuations of order at least $n^{1/8-\ep}$ in the sense of Definition \ref{lowdef}.
\end{thm}
\begin{proof}
Fix some $n$ and $\ep$. Let $V_n$ be the set of all points in $\zz^2$ that are within Euclidean distance $n^{3/4+2\ep}$ of the straight line joining $0$ and $nx$. It follows by a standard argument (as outlined, for example, in \cite[Section 6]{chatterjeefpp}) using the concentration properties of first-passage times, Alexander's rate of convergence theorem~\cite{alexander93, alexander97}, and the curvature of $B_0$ in the direction $x$,  that the geodesic from $0$ to any lattice point near $nx$ lies entirely in $V_n$ with probability tending to one as~$n\to \infty$. 

Replace the edge weights $\omega_e$ in $V_n$ by $\omega_e' := \omega_e/(1+\alpha n^{-7/8-\ep})$, where the constant $\alpha$ will be chosen later. Keep all other edge weights the same as before. Let $T_{nx}'$ be the first-passage time from $0$ to $nx$ in this new environment. Then by Corollary \ref{multcor}, it follows that
\eq{
\tv(\ml_{T_{nx}}, \ml_{T'_{nx}}) \le C\alpha,
}
where $C$ does not depend on $n$. 
On the other hand, if the original geodesic from $0$ to $[nx]$ is wholly contained in $V_n$, then 
\eq{
T'_{nx} \le \frac{T_{nx}}{1+\alpha n^{-7/8-\ep}}. 
}
Since $T_{nx}/n$ converges to a deterministic positive limit (for example, by the shape theorem, which applies to edge weight distribution in $\cp^+(1)$), this shows that there is some positive constant $c$ such that
\eq{
\lim_{n\to \infty}\pp(T_{nx}-T'_{nx}\ge cn^{1/8 - \ep}) =1.
}
Choosing $\alpha$ small enough, Lemma \ref{keylmm} completes the proof.
\end{proof}
The behavior of the set $B(t)$ is of interest in the theory of first-passage percolation. \citet{newmanpiza} defined the following exponent to measure the discrepancy of $t^{-1}B(t)$ from the limit shape $B_0$:
\eeq{
\chi' := \inf\{\kappa: (t-t^\kappa)B_0 \subseteq B(t)\subseteq (t+t^\kappa) B_0 \text{ for all large $t$ a.s.}\}.\label{npeq}
}
As far as I know, it has not been proved that $\chi'>0$ under any conditions. \citet{newmanpiza} showed that in $d=2$, $\max\{\chi', \chi\}\ge 1/5$, where $\chi$ is an exponent defined in terms of lower bounds on variances of first-passage times. To be precise, $\chi = \sup\chi_x$, where the supremum is taken over all unit vectors $x$, and
\[
\chi_x := \sup\{\gamma\ge 0: \text{for some $C>0$, $\var(T_{nx}) \ge Cn^{2\gamma}$ for all $n$}\}. 
\]
Unfortunately, the exponent $\chi$ does not contain any information about the fluctuations of $B(t)$ in the absence of matching upper bounds. The only  result that I know, that gives a true lower bound on the order of fluctuations of $B(t)$, is a theorem of \citet{zhang06}. Zhang showed that in any dimension, $B(t)$ has fluctuations of order at least $\log t$ in a certain sense, if the edge weights are Bernoulli random variables. Very recently, this result has been extended to a general class of edge weight distributions by \citet{nakajima}. Another relevant result, due to \citet{adh}, gives a lower bound on the discrepancy between the expected first-passage times and their limiting values --- but this does not say anything about fluctuations. The following theorem is the first result that shows $\chi'>0$ in two-dimensional first-passage percolation.
\begin{thm}\label{npthm}
Suppose that $d=2$ and the edge weight distribution belongs to the class $\cp^+(1)$. Let $\chi'$ be the Newman--Piza shape fluctuation exponent defined in \eqref{npeq}. Then $\chi'\ge 1/8$. 
\end{thm}
\begin{proof}
Take any $\kappa$ such that almost surely, for all large $t$,
\[
(t-t^\kappa) B_0\subseteq B(t)\subseteq (t+t^\kappa)B_0. 
\]
If $\chi' \ge 1$, there is nothing to prove. So assume that $\chi'<1$. Then we can take $\kappa < 1$. 
Let $E_t$ denote the  event in the above display, and let
\[
F_t := \bigcap_{s\ge t} E_s.
\]
From the above characterization of $\kappa$, we have
\eeq{
\lim_{t\to \infty} \pp(F_t)=1. \label{flim}
}
Take any $t>0$. Let $x$ be a direction of curvature, and let $z$ be the unique point on $\partial B_0$ in the direction $x$. Let $u$ and $v$ solve $u+u^\kappa = t$ and $v-v^\kappa = t$. Suppose that $E_v$ happens. Then 
\[
tB_0 =(v-v^\kappa) B_0 \subseteq B(v),
\]
which implies that $T_{tz}\le v$. Again, if $E_u$ happens, then 
\[
tB_0= (u+u^\kappa)B_0 \supseteq B(u),
\]
which implies that $T_{tz} \ge u$, since if $T_{tz}<u$, then we can find $t'>t$ such that $T_{t'z} < u$ due to the continuity of the map $y\mapsto T_y$ (but $t'z\not\in tB_0$, which gives a contradiction). If $F_u$ happens, then both  $E_u$ and $E_v$ happen, and so the above argument shows that $u\le T_{tz}\le v$.

Since $\kappa < 1$, a simple calculation shows that $u = t - t^\kappa + o(t^\kappa)$ and $v = t+t^\kappa + o(t^\kappa)$ as $t\to \infty$. Thus, by \eqref{flim} and the conclusion of the previous paragraph, we get
\eeq{
\lim_{t\to \infty} \pp(|T_{tz}-t|\le 2t^\kappa) =1.\label{lim1}
}
Recall that $z$ is a scalar multiple of $x$. Therefore, \eqref{lim1} implies that there is some constant $c$ such that
\[
\lim_{n\to \infty} \pp(|T_{nx} - cn|\le 2(cn)^\kappa) = 1.
\]
By Theorem \ref{fppcurv}, this is impossible unless $\kappa \ge 1/8$.
\end{proof}

\subsection{The random assignment problem}
Suppose that we have to assign $n$ tasks to $n$ workers, and $a_{ij}$ is the cost of assigning task $j$ to worker $i$. Suppose that the $a_{ij}$'s are  i.i.d.~nonnegative random variables. Let $S_n$ be the group of all permutations of $\{1,\ldots,n\}$ and let
\[
C_n := \min_{\pi\in S_n} \sum_{i=1}^n a_{i\pi(i)}.
\]
The problem of computing $C_n$ is known as the random assignment problem. \citet{aldous92} proved that if the law of the costs has a density $f$ that is nonzero and finite in a neighborhood of zero, then $C_n$ converges to a deterministic limit as $n\to \infty$. Moreover, the limit depends only on the value of $f(0)$ (assuming that $f$ is continuous at $0$). In \cite{aldous01}, Aldous proved that if $f(0)=1$, then the limit is $\zeta(2)=\pi^2/6$, confirming a conjecture of~\citet{mp85, mp87}. Later, an exact formula for $\ee(C_n)$, when the costs are exponentially distributed, was obtained by~\citet{linussonwastlund} and \citet{nairetal}. When the density of the cost distribution at zero is either blowing up to infinity or converging to zero, the situation becomes more complicated. These cases have been investigated by~\citet{wastlund12}. 

The fluctuations of $C_n$, however, are not as well understood. \citet{talagrand95} proved using his general machinery that if the costs are uniformly distributed in $[0,1]$, then the fluctuations of $C_n$ are at most of order
\[
\frac{(\log n)^2}{\sqrt{n} \log \log n}.
\]
A number of exact calculations for fluctuations are possible when the costs are exponentially distributed with mean one. Under this assumption, \citet{almsorkin} proved that the variance of $C_n$ is at least of order $1/n$, and then \citet{wastlund05, wastlund10} derived the following asymptotic formula for the variance:
\[
\var(C_n) = \frac{4\zeta(2)-4\zeta(3)}{n} + O\biggl(\frac{1}{n^2}\biggr).
\]
\citet{wastlund05, wastlund10} also gave formulas for higher moments when the costs are exponentially distributed, but the asymptotics of these formulas are hard to understand. There is, however, a more direct way to extract lower bounds on fluctuations when the costs are exponentially distributed, using the memoryless property of the exponential distribution. This has been worked out in a related model by~\citet{hesslerwastlund}, also appearing in the Ph.D.~thesis of \citet{hessler}.

When the costs are not exponential, nothing is known about lower bounds on the fluctuations of $C_n$. The following theorem gives a general lower bound of order $n^{-1/2}$, which, in the face of the evidence presented above, appears to be the correct order (recall that an upper bound of order $n^{-1/2}$ is not yet known for non-exponential costs). The proof uses Lemma \ref{keylmm}, but the simple multiplicative perturbation of Section \ref{pertsec} does not give the correct answer for this problem. Instead, a more complicated coupling is used.  
\begin{thm}
Suppose that the cost distribution belongs to the class $\cp^+(1)$. Then the optimal cost $C_n$ in the random assignment problem has fluctuations of order at least $n^{-1/2}$, in the sense of Definition~\ref{lowdef}.
\end{thm}
\begin{proof}
Throughout this proof, $C$ will denote any positive constant that may depend only on the cost distribution and nothing else. The value of $C$ may change from line to line or even within a line.

 Define a function $\phi:[0,\infty)\to [0,\infty)$ as
\eq{
\phi(x) &= 
\begin{cases}
\sqrt{n} x &\text{ if } 0\le x\le 1/n,\\
x + 1/\sqrt{n}-1/n &\text{ if } x> 1/n.
\end{cases}
}
Then $\phi$ is absolutely continuous and strictly increasing, and $\phi(0)=0$. Moreover, $\phi(x)\le x+1$ for all $x$, and  
\eeq{
\phi'(x)&=
\begin{cases}
\sqrt{n}  &\text{ if } 0< x< 1/n,\\
1 &\text{ if } x> 1/n.
\end{cases}
\label{phiprime}
}
Let $e^{-V}$ be the density function of the cost distribution. Let $X$ be a nonnegative random variable with probability density $e^{-V}$. Let $Y$ solve 
\[
Y+\alpha n^{-1} \phi(Y) = X,
\]
where $\alpha$ is a positive constant that will be chosen later. 
Since the map $x\mapsto x + \alpha n^{-1} \phi(x)$ is strictly increasing and continuous on $[0,\infty)$ and sends $0$ to $0$, it is a bijection of $[0,\infty)$ onto itself. Therefore, $Y$ is uniquely defined. The probability density function of $Y$ is 
\[
(1+\alpha n^{-1}\phi'(x)) e^{-V(x+\alpha n^{-1}\phi(x))}. 
\]
For $\ep \in (-1/2, 1/2)$, define
\eq{
g(\ep) := \int_0^\infty \sqrt{(1+\ep\phi'(x))e^{-V(x+\ep \phi(x))}e^{-V(x)}} \,d x,
}
so that
\eq{
\rho(\ml_X, \ml_Y) = g(\alpha n^{-1}). 
}
As in the proof of Theorem \ref{multthm}, it is easy to verify using the assumptions on $V$ and $\phi$ that $g$ is a $C^\infty$ function of $\ep$ and the derivatives with respect to $\ep$ can be taken inside the integral. Note that $g(0) = 1$, and 
\eq{
g'(\ep) &= \int_0^\infty \frac{\phi'(x)}{2\sqrt{1+\ep\phi'(x)}}e^{-(V(x+\ep \phi(x))+V(x))/2} \,d x\\
&\qquad - \frac{1}{2}\int_0^\infty \sqrt{1+\ep\phi'(x)}V'(x+\ep\phi(x))\phi(x)e^{-(V(x+\ep \phi(x))+V(x))/2} \,d x.
}
Thus,
\eq{
g'(0) &= \frac{1}{2}\int_0^\infty(\phi'(x)-V'(x)\phi(x)) e^{-V(x)} \, d x,
}
which equals zero by integration by parts, since $\phi(0)=0$. Next, note that
\eq{
&g''(\ep) = -\int_0^\infty \frac{\phi'(x)^2}{4(1+\ep\phi'(x))^{3/2}}e^{-(V(x+\ep \phi(x))+V(x))/2} \,d x\\
&\quad -\frac{1}{2}\int_0^\infty \frac{\phi'(x)}{\sqrt{1+\ep\phi'(x)}}V'(x+\ep\phi(x))\phi(x)e^{-(V(x+\ep \phi(x))+V(x))/2} \,d x\\
&\quad - \frac{1}{2}\int_0^\infty \sqrt{1+\ep\phi'(x)}V''(x+\ep\phi(x))\phi(x)^2e^{-(V(x+\ep \phi(x))+V(x))/2} \,d x\\
&\quad + \frac{1}{4}\int_0^\infty \sqrt{1+\ep\phi'(x)}V'(x+\ep\phi(x))^2\phi(x)^2e^{-(V(x+\ep \phi(x))+V(x))/2} \,d x.
}
Because of \eqref{phiprime}, it is convenient to write each of the four terms in the above display as a sum of two integrals --- one from $0$ to $1/n$ and another from $1/n$ to $\infty$. Due to \eqref{phiprime} and the boundedness of $V$ near zero, the magnitude of the first part is bounded by $C$ in all four cases. In the second part, $\phi'(x)$ is bounded by $1$ and $\phi(x)$ is bounded by $x+1$. Therefore using the properties of $V$ and a few applications of the Cauchy--Schwarz inequality and change-of-variables, we see that the magnitude of the second part is also bounded by $C$ in all four cases. Thus, 
\[
\sup_{-1/2<\ep<1/2} |g''(\ep)|\le C.
\]
As a consequence, 
\eeq{
\rho(\ml_X, \ml_Y) &= g(\alpha n^{-1}) \ge 1 - C\alpha^2 n^{-2}. \label{bcbd}
}
Now for each $i,j$, let $a'_{ij}$ solve 
\eq{
a'_{ij} +\alpha n^{-1}\phi(a'_{ij}) = a_{ij}. 
}
Let $C_n'$ be the optimal assignment cost with these new costs. By \eqref{bcbd} and Lemma \ref{tvprod}, it follows that
\eeq{
\tv(\ml_{C_n},\ml_{C_n'}) &\le C\alpha.\label{tvcn}
}
For $1\le i\le n$, let
\eq{
b_i := \min_{1\le j\le n} a_{ij}. 
}
Since $e^{-V}$ is a bounded density, 
\eeq{
\pp(b_i \ge 1/n) &\ge (1-C/n)^n\ge K, \label{biq}
}
where $K$ is some positive constant that does not depend on $n$.
Let 
\[
A := \{i: b_i \ge 1/n\}.
\]
Since $b_1,\ldots, b_n$ are i.i.d.~random variables, \eqref{biq} shows that 
\eeq{
\lim_{n\to \infty}\pp(|A|\ge Kn/2) =1.\label{akn2}
} 
Now notice that $a_{ij}'\le a_{ij}$ for all $i$ and $j$. Moreover, if $i\in A$, then for any $j$, $a_{ij}\ge 1/n$. Since $x\mapsto x + n^{-1}\phi(x)$ is an increasing map, this implies that $a_{ij}' \ge x_n$, where $x_n$ is the unique solution of
\[
x_n + \alpha n^{-1}\phi(x_n)= n^{-1}.
\]
A simple calculation shows that
\[
x_n = \frac{1}{n+\alpha\sqrt{n}}. 
\]
Thus, if $i\in A$, then for any $j$,
\eq{
a_{ij}' &\ge \frac{1}{n+\alpha\sqrt{n}},
}
and therefore
\eq{
a_{ij}-a_{ij}' &= \alpha n^{-1}\phi(a_{ij}')\\
&\ge \alpha n^{-1} \phi\biggl(\frac{1}{n+\alpha\sqrt{n}}\biggr) = \frac{\alpha }{n^{3/2} + \alpha n}. 
}
Combining all observations, we get
\eq{
C_n - C_n' &\ge \frac{\alpha |A|}{n^{3/2} + \alpha n}. 
}
By \eqref{tvcn}, \eqref{akn2} and Lemma~\ref{keylmm}, this completes the proof. 
\end{proof}

\subsection{Determinants of random matrices}
Let $n$ and $N$ be two positive integers, and let $f$ be a measurable function from $\rr^n$ into $\rr^{N\times N}$, the set of all $N\times N$ real matrices. Suppose that there is some $r> 0$ such that for all $x_1,\ldots, x_n\in\rr$ and all $\lambda\ge0$, 
\eeq{
f(\lambda x_1,\ldots, \lambda x_n) = \lambda^r f(x_1,\ldots, x_n). \label{homog2}
}
Let $X_1,\ldots, X_n$ be i.i.d.~random variables. Then $M= f(X_1,\ldots, X_n)$ is an $N\times N$ random matrix. Many common families of random matrices, such as Wigner matrices, sample covariance matrices, random Toeplitz and Hankel matrices, and random band matrices, can be obtained in the above manner as functions of independent random variables where the function satisfies~\eqref{homog2} for some $r$. For example, for a Wigner matrix of order $N$, $r=1$ and $n = N(N+1)/2$. The following theorem gives a lower bound on the order of fluctuations of $\log |\det M|$. 
\begin{thm}\label{matrixthm}
Let $X_1,X_2,\ldots$ be a sequence of i.i.d.~random variables with probability density in either $\cp(1)$ or $\cp^+(1)$. For each $n$, let $N_n$ be a  positive integer and let $f_n$ be a measurable function from $\rr^n$ into $\rr^{N_n\times N_n}$ satisfying~\eqref{homog2} for some fixed $r>0$. Let $M_n := f_n(X_1,\ldots, X_n)$. Then $\log|\det M_n|$ has fluctuations of order at least $n^{-1/2} N_n$, in the sense of Definition \ref{lowdef}.
\end{thm}
\begin{proof}
Throughout this proof, $C$ will denote any constant that may depend only on the distribution of the $X_i$'s, but not on $n$.

Fix $n$. Let $X_i' := X_i/(1+\alpha n^{-1/2})$, where $\alpha$ will be determined later. Let $M_n' := f_n(X_1',\ldots,X_n')$. Let $L_n := \log |\det M_n|$ and $L_n':= \log |\det M_n'|$. By Corollary \ref{multcor},
\eq{
\tv(\ml_{L_n}, \ml_{L_n'}) \le \tv(\ml_{(X_1,\ldots, X_n)}, \ml_{(X_1',\ldots, X_n')})\le C\alpha. 
}
On the other hand, by \eqref{homog2}, 
\eq{
\det M_n' = (1+\alpha n^{-1/2})^{rN_n} \det M_n,
}
and hence
\eq{
L_n' = L_n +r N_n \log(1+\alpha n^{-1/2}). 
}
An application of Lemma \ref{keylmm} completes the proof.
\end{proof}
Consider now the special case of sample covariance matrices. Let $n$ and $p$ be two positive integers, and let $Y_1,\ldots, Y_n$ be i.i.d.~$p$-dimensional random vectors, whose components are i.i.d.~random variables. The sample covariance matrix for this data is defined as
\eeq{
W := \frac{1}{n}\sum_{i=1}^p (Y_i - \bar{Y})(Y_i- \bar{Y})^T,\label{wdef}
}
where $x^T$ denotes the transpose of a column vector $x$, and 
\[
\bar{Y} := \frac{1}{n}\sum_{i=1}^n Y_i.
\]
Note that this is slightly different than the Wishart matrices usually considered in random matrix theory, which have the same definition, except that $\bar{Y}$ is not subtracted from $Y_i$ in \eqref{wdef}. 

A sample covariance matrix is positive semi-definite and so its determinant is always nonnegative. The logarithm of the determinant of a sample covariance matrix is an important object in statistics, where it is used to perform tests of hypotheses~\cite{anderson03}. A central limit theorem for $\log \det W$, when $p$ is fixed, $n\to \infty$ and the $Y_i$'s are complex Gaussian random vectors, was established in 1963 by \citet{goodman63}. The high dimensional case, where $p$ and $n$ both tend to infinity, was solved by \citet{cailiangzhou} in 2015. (The corresponding result for Wishart matrices, however, is standard fare in random matrix theory --- see, for example, \cite[Chapter 7]{pasturshcherbina}.) The main result of~\cite{cailiangzhou} is a central limit theorem for $\log \det W$ when the $Y_i$'s are real Gaussian random vectors (possibly with correlated coordinates), $n\to \infty$, and $p$ is allowed to vary arbitrarily but under the constraint that $p\le n$ (so that $\det W \ne 0$). In this scenario, \citet{cailiangzhou} show that $\log \det W$ has Gaussian fluctuations of order $\sqrt{p/n}$ if $p/n\to r\in [0,1)$, and of order $\sqrt{\log n}$ if $p/n\to 1$. 

The non-Gaussian case is open. In particular, a central limit theorem for $\log \det W$ has not been proved in the setting described above, that is, $Y_i$'s having i.i.d.~but not necessarily Gaussian coordinates. The following corollary of Theorem \ref{matrixthm} provides a lower bound on the fluctuation of $\log \det W$ which appears to be of the correct order if $p/n \to r\in [0,1)$, in view of the result of \citet{cailiangzhou}. 
\begin{cor}\label{matrixcor}
Let $Y_1,\ldots, Y_n$ be i.i.d.~$p$-dimensional random vectors and let $W$ be defined as in \eqref{wdef}. Suppose that the coordinates of $Y_i$'s are i.i.d.~with probability density in $\cp(1)$ (which remains fixed as $n$ and $p$ vary). Then as $n\to \infty$ and $p$ varies arbitrarily as a function of $n$ (with the constraint that $p\le n$), $\log \det W$ has fluctuations of order at least $\sqrt{p/n}$, in the sense of Definition \ref{lowdef}.
\end{cor}
\begin{proof}
Note that the $p\times p$ matrix $W$ is a function of $np$ i.i.d.~random variables, and this function satisfies \eqref{homog2} with $r=2$. Thus, the $n$ in Theorem~\ref{matrixthm} should be replaced by $np$ and $N_n$ should be replaced by $p$. With these replacements, the lower bound on the order of fluctuations turns out to be $(np)^{-1/2}p = \sqrt{p/n}$, proving the claim. 
\end{proof}

It is surprising to me that a soft technique based on Lemma~\ref{keylmm} and Corollary~\ref{multcor} can actually yield the correct lower bound in Corollary~\ref{matrixcor}. Indeed, the method does not yield the correct bound for Wigner matrices. Improving on an earlier work of \citet{taovu12}, \citet{nguyenvu} proved the central limit theorem for log-determinants of Wigner matrices with non-Gaussian entries in 2014. According to this result, the log-determinant has fluctuations are of order $\sqrt{\log n}$. A straightforward application of Theorem~\ref{matrixthm}, however, only gives a lower bound of order $1$. Possibly a better coupling than the one provided by Corollary \ref{multcor} is needed to achieve the $\sqrt{\log n}$ lower bound.

\section{Open problems}
There are many open questions about lower bounds for fluctuations of random variables. Here is a list of questions that are closely associated with the examples worked out in this paper.
\begin{enumerate}
\item Extend the results of this paper beyond the distribution classes $\cp(d)$ and $\cp^+(d)$. In particular, proofs under minimal assumptions would be very desirable. 
\item In the traveling salesman problem for uniformly distributed points on $[0,1]^2$, prove that the variance of the length of the optimal tour converges to a constant as the number of points tends to infinity, and identify this constant if possible. This conjecture is due to Mike Steele, who told me about it in a personal communication. 
\item Prove a tight lower bound for the fluctuations of the length of the minimal matching when the points are uniformly distributed in $[0,1]^d$. 
\item Prove a tight lower bound for fluctuations in the longest common subsequence problem for random words. Considerable progress on this problem has been made in~\cite{gongetal, houdrema, houdrematzinger, lembermatzinger}, but the most important case of uniformly distributed letters is open. A solution of this problem would complete the proof of the central limit theorem for longest common subsequences, as shown by~\citet{houdreislak}.
\item Improve the lower bound for the fluctuations of the first-passage time in two-dimensional first-passage percolation.
\item Extend the lower bound result for first-passage percolation to the case of discrete edge weights. 
\item Prove any nontrivial lower bound for the fluctuations of the first-passage time in higher dimensions (a lower bound of order $1$ is easy using the method of this paper).
\item Improve the lower bound for the Newman--Piza exponent $\chi'$ for two-dimensional first-passage percolation.
\item Show that $\chi'>0$ in higher dimensions.
\item Improve the lower bound on the order of fluctuations of the free energy  of the S-K model or show that it is optimal. Same for the ground state energy.
\item Prove the optimality of the lower bound in the random assignment problem for cost distributions in $\cp^+(1)$ or any other general class of cost distributions. 
\item Prove tight lower bounds for fluctuations of functionals of random matrices other than the determinant, such as linear statistics of eigenvalues and the maximum and minimum eigenvalues, in ensembles where such results are not known. This includes a variety of patterned random matrix ensembles, such as random Toeplitz and Hankel matrices and random band matrices. For a survey of results on patterned random matrices, see~\cite{boseetal}. 
\item Prove a distributional limit theorem in any of the examples discussed in this paper.  
\end{enumerate}

\section*{Acknowledgments} 
I am grateful to David Aldous, Louis-Pierre Arguin, Erik Bates, Wei-Kuo Chen, Michael Damron, Peter Forrester, Christian Houdr\'e, Svante Janson, Ron Peled,  Yuval Peres, Mike Steele, Johan W\"astlund, Harry Zhou and the anonymous referees for many helpful comments and references.

\end{document}